\documentclass{amsart}

\title{Remarks on the solution map for Yudovich solutions of the Euler equations}

\author{Huy Q. Nguyen}
\address{Department of Mathematics, University of Maryland, College Park, MD 20742, USA}
\email{hnguye90@umd.edu}

\usepackage[margin=1.5in]{geometry}
\usepackage{amsmath, amsthm, amssymb, mathrsfs}
\usepackage{times}
\usepackage{color}
\usepackage{hyperref}

\usepackage{graphicx}
\usepackage{tikz}
\usetikzlibrary{decorations.pathmorphing}

\newcommand{\bq}{\begin{equation}}
\newcommand{\eq}{\end{equation}}
\newcommand{\bqa}{\begin{eqnarray*}}
\newcommand{\eqa}{\end{eqnarray*}}

\theoremstyle{plain}
\newtheorem{theo}{Theorem}[section]

\newtheorem{lemm}[theo]{Lemma}

\theoremstyle{definition}
\newtheorem{rema}[theo]{Remark} 

\DeclareMathOperator{\supp}{supp}

\DeclareSymbolFont{pletters}{OT1}{cmr}{m}{sl}
\DeclareMathSymbol{s}{\mathalpha}{pletters}{`s}

\def\eps{\varepsilon}
\def\na{\nabla}

\def\tdm{\frac{3}{2}}

\def\Rr{\mathbb{R}}

\def\p{\partial}

\def\na{\nabla}
\def\wsc{\rightharpoonup}

\def\ol{\overline}
\def\T{\mathbb{T}}

\def\Om{\Omega}

\numberwithin{equation}{section}
\def\om{\omega}

\def\wsc{\overset{\ast}{\rightharpoonup}}
\def\wc{\rightharpoonup}

\date{\today}
\begin{document}
\begin{abstract}
Consider Yudovich solutions to the incompressible Euler equations with bounded initial vorticity in bounded planar domains. We present a purely Lagrangian proof  that  the solution map is  strongly continuous in $L^p$ for all $p\in [1, \infty)$ and is weakly-$*$ continuous in $L^\infty$.
\end{abstract}



\maketitle

\section{Introduction}
Let $\Om$ be a bounded domain in $\Rr^2$ with $C^2$ boundary $\p\Om$. Let $\Delta_D$ denote the Dirichlet Laplacian associated to $\Om$. The vorticity formulation of the 2D incompressible Euler equations  is 
\bq\label{Euler}
\p_t \om+u\cdot \na \om=0,\quad(x, t)\in \Om\times (0, \infty),
\eq
where the velocity $u$ is recovered from the vorticity $\om$ through the Biot-Savart law 
\bq\label{BS}
u(x)=\na^\perp \Delta_D^{-1}\om=K*\om.
\eq
Note that $u$ given by \eqref{BS} is parallel to the boundary $\p\Om$. In the celebrated work \cite{Yu}, Yudovich proved the existence and uniqueness of global solutions to \eqref{Euler} with bounded initial vorticity. This theory includes the important class of vortex patches \cite{MB}.  
\begin{theo}[\protect{\cite{Yu, MP}}]\label{Yu}
Let $\om_0\in L^\infty(\Om)$. There exists a unique triple $(\om, u, X_t)$ solution to \eqref{Euler}  such that $\om\in L^\infty(\Rr; L^\infty(\Om))$,  $u(t)=K*\om(t)$, $X_t:\Om\to\Om$  measure-preserving, invertible and
\begin{align}\label{def:X}
&\frac{d}{dt}X_t(x)=u(X_t(x, t), t),\quad X_0(x)=x\quad\forall x\in\ol{\Om},\\
\label{form:om:0}
&\om(x, t)=\om_0(X^{-1}_t(x)).
\end{align}
Moreover, the flow $X_t:\ol{\Om}\to \ol{\Om}$ is H\"older continuous on $\ol{\Om}$ with exponent $\exp(-C|t|\|\om_0\|_{L^\infty(\Om)})$ for some $C=C(\Om)$. 
\end{theo}
The preceding version of Yudovich theory is taken from \cite{MP} and is elegant in that the notion of solution is naturally defined in terms of the Lagrangian flow and does not involve test functions. The purpose of this note is to present proofs of folklore about the continuity of the solution map for Yudovich solutions in this purely Lagrangian framework.

To define the inverse of the flow $X_t$, we  let  $X_{s, t}(x)$ be the solution of 
\bq\label{def:Xst}
\frac{d}{dt}X_{s, t}(x)=u(X_{s, t}(x), t),\quad X_{s, s}(x)=x.
\eq
In view of \eqref{def:X}, we denote $X_{0, t}\equiv X_t$. Then we have
\[
X_t^{-1}=X_{t, 0}
\]
 and   \eqref{form:om:0} becomes
\bq\label{form:om}
\om(x, t)=\om_0(X_{t, 0}(x)).
\eq
  We first state the continuity in time of Yudovich solutions.
\begin{lemm}\label{lemm}
For all initial data $\om_0\in L^\infty(\Om)$, the unique solution $\om$  given by Theorem \ref{Yu} belongs to $C(\Rr; L^p(\Om))\cap C_w(\Rr; L^\infty(\Om))$ for all $p\in [1, \infty)$. Here $C_w(\Rr; L^\infty(\Om))$ denotes the space of functions that are continuous in time with values in the weak-$*$ topology of $L^\infty(\Om)$.
\end{lemm}
\begin{proof}
We first note that since   the velocity field $u$ is Log-Lipschitz (see \eqref{LL}), $X_t(x)\in C(\overline{\Om}\times \Rr)$ (see \cite{MP}). Therefore, if $\om_0\in C(\ol{\Om})$ then it is clear that $\om\in C(\Rr; L^p(\Om))$ for all $p\in [1, \infty]$. For $p\in [1, \infty)$ and $\om_0\in L^\infty(\Om)\subset L^p(\Om)$, using the fact that $C(\ol{\Om})$ is dense in $L^p(\Om)$ and $X_t$ is measure-preserving, we obtain $\om \in C(\Rr; L^p(\Om))$.

 For any $f\in C(\ol{\Om})$,  \eqref{form:om:0} yields
\[
g(t):=\int_\Om \om(x, t)f(x)dx=\int_\Om \om_0(X^{-1}_t(x))f(x)dx=\int_\Om \om_0(x)f(X_t(x))dx.
\]
Thus $g\in C(\Rr)$ since $X_t(x)\in C(\ol{\Om}\times \Rr)$. Since  $C(\ol{\Om})$ is dense in $L^1(\Om)$ and $X_t$ is measure-preserving, it follows that $t\mapsto \int_\Om \om(x, t)f(x)dx$ is continuous for all $f\in L^1(\Om)$. Therefore, $\om\in C_w(\Rr; L^\infty(\Om))$.
\end{proof}
 By virtue of Lemma \ref{lemm}, for every $t>0$, the solution map 
\bq\label{def:S}
S_t: L^\infty(\Om)\ni \om_0\mapsto \om(t) \in L^\infty(\Om)
\eq
is well defined. We prove that  $S_t$ is strongly continuous in $L^p(\Om)$ for all $p\in [1, \infty)$.
\begin{theo}\label{theo}
Let $p\in [1, \infty)$. Let $\om_0$, $\om_0^n\in L^\infty(\Om)$ such that $(\om_0^n)_n$  converges to $\om_0$ in $L^p(\Om)$. Then for all $T>0$  we have
\bq\label{conv}
\lim_{n\to \infty}\sup_{t\in [-T, T]}\| S_t(\om_0^n)-S_t(\om_0)\|_{L^p(\Om)}=0.
\eq
\end{theo}
Moreover, $S_t$ is  continuous in the weak-$*$ topology of $L^\infty(\Om)$.
\begin{theo}\label{theo:1}
 If $\om_0^n\overset{\ast}{\rightharpoonup}  \om_0$ in $L^\infty(\Om)$, then $S_\cdot(\om_0^n)\wsc S_\cdot(\om_0)$ in $L^\infty(\Om\times (-T, T))$ for all $T>0$ and $S_t(\om_0^n)\wsc S_t(\om_0)$ in  $L^\infty(\Om)$ for all $t\in \Rr$. 
 \end{theo}
It was obtained in \cite[Corollary 1]{CDE} that for the torus $\T^2$,  the solution map for Yudovich solutions is continuous in $L^p$ {\it when restricted to bounded sets of $L^\infty$}. Theorem \ref{theo}  dispenses with the restriction to bounded sets of $L^\infty$ and holds for domains with boundary. The proof in \cite{CDE} is Eulerian and relies on  $L^2$ energy estimates for the velocity and vorticity differences. On the other hand, our proof of Theorem \ref{theo}  is purely Lagrangian: $L^p$ estimates for the vorticity difference is deduced from an $L^1$ estimate for the difference of the flow maps. The latter is established by employing an idea in \cite{MP} for the uniqueness of Yudovich solutions. We remark that Theorem \ref{theo:1} is stated without proof in \cite{Sverak} and is used to deduce properties of the omega-limit set of  the 2D Euler equations.

 On the whole space $\Om=\Rr^2$, the same statement in Theorem \ref{Yu} holds with $L^\infty(\Om)$ replaced by $L^\infty_c(\Rr^2)$, the space of $L^\infty(\Rr^2)$ functions with compact support. Note however that the flow map $X_t$ is then only {\it locally} H\"older continuous with exponent $\exp(-C|t|\|\om_0\|_{L^1\cap L^\infty})$, where $C$ is a universal constant and $L^1\cap L^\infty\equiv L^1(\Rr^2)\cap L^\infty(\Rr^2)$ is equipped with the norm
 \[
 \| \cdot\|_{L^1\cap L^\infty}=\| \cdot\|_{L^1(\Rr^2)}+\|\cdot\|_{L^\infty(\Rr^2)}.
 \]
  We have the following version of Theorem \ref{theo:1}.
 \begin{theo}\label{theo:2}
 Assume that $\om_0 \in L^\infty_c(\Rr^2)$, $(\om_0^n)_n\subset L^\infty_c(\Rr^2)$ is bounded in $L^1(\Rr^2)$ and  $\om_0^n\wsc \om_0$ in $L^\infty(\Rr^2)$. Then   $S_\cdot(\om_0^n)\wsc S_\cdot(\om_0)$ in $L^\infty(\Om\times (-T, T))$ for all $T>0$ and $S_t(\om_0^n)\wsc S_t(\om_0)$ in $L^\infty(\Rr^2)$  for all $t\in \Rr$.
\end{theo}
Under the hypothesis of Theorem \ref{theo:2}, for all $p\in [1, \infty]$ we have that $\| S_t(\om^n_0)\|_{L^p}=\| \om_0^n\|_{L^p}$ is uniformly bounded  by interpolation. Therefore, the conclusion in Theorem \ref{theo:2}  implies $S_t(\om_0^n)\wsc S_t(\om_0)$ in $\mathcal{M}(\Rr^2)$, the space of signed Radon measures on $\Rr^2$, and $S_t(\om_0^n)\wc S_t(\om_0)$ in $L^p(\Om)$ for all $p\in (1, \infty)$. 
\begin{rema}
As we have mentioned earlier, the notion of solution in Theorem \ref{Yu} does not involve test functions. On the other hand, if $\om$ is such a solution, then for any $\phi\in C^1(\ol{\Om}\times [t_1, t_2])$ we have
\[
\begin{aligned}
&\int_{t_1}^{t_2}\int_\Om\om(x, t)\p_t\phi(x, t)dxdt\\
&=\int_{t_1}^{t_2}\int_\Om\om_0(x)(\p_t\phi)(X_t(x), t)dxdt\\
&=\int_{t_1}^{t_2}\int_\Om\om_0(x)\p_t[\phi(X_t(x), t)]dxdt-\int_{t_1}^{t_2}\int_\Om\om_0(x)\na\phi(X_t(x), t)\cdot \p_t X_t(x)dxdt\\
&=\int_\Om\om_0(x)[\phi(X_{t_2}(x), {t_2})-\phi(X_{t_1}(x), {t_1})]dx-\int_{t_1}^{t_2}\int_\Om\om_0(x)\na\phi(X_t(x), t)\cdot u(X_t(x), t)dxdt\\
&=\int_\Om\om(x, t_2)\phi(x, t_2)-\om(x, t_1)\phi(x, {t_1})dx-\int_{t_1}^{t_2}\int_\Om\om(x, t) u(x, t)\cdot \na\phi(x, t)dxdt.
\end{aligned}
\]
Thus $\om$ obeys the weak form 
\bq\label{weakform}
\int_{t_1}^{t_2}\int_\Om\om(x, t)\left[\p_t\phi(x, t)+ u(x, t)\cdot \na\phi(x, t)\right]dxdt=\int_\Om\om(x, t_2)\phi(x, t_2)-\om(x, t_1)\phi(x, {t_1})dx.
\eq
\end{rema}
\section{Proof of Theorem \ref{theo}}
We first recall the following estimates for the Biot-Savart kernel $K$.
\begin{lemm}[\protect{\cite{Yu}}]
There exists $C$ depending only on $\Om$ such that for all $x$, $y$, $a$, $b\in \Om$, we have
\begin{align}\label{estK:1}
&|K(x, y)|\le C|x-y|^{-1},\\ \label{estK:2}
&\int_\Om |K(x, a)-K(x, b)|dx\le C\varphi(|a-b|),
\end{align}
where $\varphi$ is the Log-Lipschitz modulus of continuity
\bq
\varphi(r)=r(1-\ln r)\quad\text{if } 0<r\le 1,\quad \varphi(r)=1\quad\text{if } r> 1.
\eq
\end{lemm}
As a direct consequence of \eqref{estK:1} and \eqref{estK:2}, if $\om\in L^\infty(\Om)$ then $u=K*\om$ is bounded and Log-Lipschitz:
\begin{align}\label{u:uni}
&\| u\|_{L^\infty(\Om)}\le C\|\om\|_{L^\infty(\Om)},\\
\label{LL}
&|u(x)-u(y)|\le C\|\om\|_{L^\infty}\varphi(|x-y|)\quad\forall x,~y\in \Om.
\end{align}
Let $(\om^j, u^j, X^j_t)$, $j=1, 2$ be two solutions of \eqref{Euler} with initial data $\om_0^j\in L^\infty(\Om)$.
For notational simplicity we shall write $L^p\equiv L^p(\Om)$. Fix $p\in [1, \infty)$. We have the elementary inequalities 
\bq
(a+b)^p\le 2^{p-1}(a^p+b^p),\quad (a+b+c)^p\le 2^{p-1}a^p+2^{2p-2}(b^p+c^p)\quad\forall a, b, c\in \Rr_+.
\eq
Since the flow maps $X^j_{s, t}$ are measure-preserving, we have
 \bq\label{est:0}\begin{aligned}
 \| \om^1(t)-\om^2(t)\|_{L^p}^p&=\int_\Om|\om_0^1(X^1_{t, 0}(x))-\om_0^2(X^2_{t, 0}(x))|^pdx\\
 &\le 2^{p-1}\int_\Om|\om_0^1(X^1_{t, 0}(x))-\om_0^1(X^2_{t, 0}(x))|^pdx\\
 &\quad+2^{p-1}\int_\Om|\om_0^1(X^2_{t, 0}(x))-\om_0^2(X^2_{t, 0}(x))|^pdx\\
 & \le 2^{p-1}\int_\Om|\om_0^1(X^1_{t, 0}(x))-\om_0^1(X^2_{t, 0}(x))|^pdx+ 2^{p-1}\| \om^1_0-\om^2_0\|_{L^p}^p.
 \end{aligned}
 \eq
We extend $\om_0^1$ to zero outside $\Om$ and  approximate $\om_0^1$ by $\om_0^1*\rho_\eps$, where $\rho_\eps$ is the standard mollifier. It follows from \eqref{est:0} that
\bq\label{est:1}
\begin{aligned}
 \| \om^1(t)-\om^2(t)\|_{L^p}^p&\le 2^{p-1}\| \om^1_0-\om^2_0\|_{L^p}^p+2^{3p-2}\| \om_0^1*\rho_\eps-\om_0^1\|_{L^1}\\
 &\quad+2^{2p-2}\int_\Om|(\om_0^1*\rho_\eps)(X^1_{t, 0}(x))-(\om_0^1*\rho_\eps)(X^2_{t, 0}(x))|^pdx\\
& \le 2^{p-1}\| \om^1_0-\om^2_0\|_{L^p}^p+2^{3p-2}\| \om_0^1*\rho_\eps-\om_0^1\|_{L^1}\\
&\quad+2^{2p-2}\|\om_0^1*\rho_\eps\|^p_{\dot C^{1/p}}\int_\Om |X^1_{t, 0}(x)-X^2_{t, 0}(x)|dx.
 \end{aligned}
 \eq
 Set $F(x, t, r):=|X^1_{t, r}(x)-X^2_{t, r}(x)|$. Integrating \eqref{def:Xst} with respect to $t$ we deduce
 \bq\label{split:F}
 \begin{aligned}
 F(x, t, r)
 &\le \left|\int_t^r |u^1(X^1_{t, s}(x), s)-u^1(X^2_{t, s}(x), s)|ds\right|+\left| \int_t^r|u^1(X^2_{t, s}(x), s)-u^2(X^2_{t, s}(x), s)|ds\right|\\
 &:=I_1(x, t, r)+I_2(x, t, r).
   \end{aligned}
   \eq
   The Log-Lipschitz bound \eqref{LL} yields
   \bq\label{est:I1}
   |I_1(x, t, r)|\le C\|\om_0^1\|_{L^\infty}\left|\int_t^r\varphi(F(x, t, s))ds\right|.
   \eq
   As for $I_2$ we use the definition $u^j=K*\om^j$, \eqref{form:om} together the fact that the maps $X^j_{s, t}$ are measure-preserving, giving
   \begin{align*}
 |I_2(x, t, r)|&\le \left|\int_t^r\left|\int_\Om K(X^2_{t, s}(x), y)\om_0^1(X^1_{s, 0}(y))-K(X^2_{t, s}(x), y)\om_0^2(X^2_{s, 0}(y))dy\right|ds\right|\\
 &=\left|\int_t^r\left|\int_\Om |K(X^2_{t, s}(x), X^1_{s}(y))\om_0^1(y)-K(X^2_{t, s}(x), X^2_{s}(y))\om_0^2(y)dy\right|ds\right|\\
 &\le \left|\int_t^r\int_\Om |K(X^2_{t, s}(x), X^1_{s}(y))-K(X^2_{t, s}(x), X^2_{s}(y))||\om_0^1(y)|dyds\right|\\
 &\quad +\left|\int_t^r\int_\Om |K(X^2_{t, s}(x), X^2_{s}(y))|\om_0^1(y)-\om_0^2(y)|dyds\right|:=I_2^a+I_2^b.
\end{align*}
Integrating $I_2$ in $x$ and using the fact that $X^2_{t, s}$ is measure-preserving, we deduce
\[
\begin{aligned}
\int_\Om |I_2^a(x, t, r)|dx&=\left|\int_t^r\int_\Om\int_\Om |K(x, X^1_{s}(y))-K(x, X^2_{s}(y))|dx|\om_0^1(y)|dyds\right|\\
&\le C\|\om_0^1\|_{L^\infty}\left|\int_t^r\int_\Om \varphi(F(y, 0, s))dyds\right|,
\end{aligned}
\]
where we have used \eqref{estK:2} in the second estimate. Since $\varphi$ is concave, Jensen's inequality implies
\bq\label{est:I2a}
\frac{1}{|\Om|}\int_\Om |I_2^a(x, t, r)|dx\le C\|\om_0^1\|_{L^\infty}\left|\int_t^r\varphi\left(\frac{1}{|\Om|}\int_\Om F(y, 0, s)dy\right)ds\right|.
\eq
On the other hand, \eqref{estK:1} gives
\bq\label{est:I2b}
\begin{aligned}
\int_\Om |I_2^b(x, t, r)|dx&\le C\left|\int_t^r\int_\Om\int_\Om \frac{1}{|X_{t, s}^2(x)-X^2_{s}(y)|}dx |\om_0^1(y)-\om_0^2(y)|dyds\right|\\
&=C\left|\int_t^r\int_\Om\int_\Om \frac{1}{|x-X^2_{ s}(y)|}dx |\om_0^1(y)-\om_0^2(y)|dyds\right|\\
&\le C|t-r| \|\om_0^1-\om_0^2\|_{L^1}.
\end{aligned}
\eq
A combination of \eqref{split:F}, \eqref{est:I1}, \eqref{est:I2a} and \eqref{est:I2b} implies that $\eta(t, r):=|\Om|^{-1}\int_\Om F(x, t, r)dx$ satisfies 
\bq\label{growth:eta}
\eta(t, r)\le  C|t-r|\| \om_0^1-\om_0^2\|_{L^1} +C\|\om_0^1\|_{L^\infty}\left\{\left|\int_t^r \varphi(\eta(t, s))ds\right|+\left|\int_t^r \varphi(\eta(0, s))ds\right|\right\}
\eq
for all $t, r\in \Rr$, where $C=C(\Om)$. 

Let $T>0$ be arbitrary. We first consider $t=0$ in \eqref{growth:eta}. We choose $\om_0^j$ such that $CT\| \om_0^1-\om_0^2\|_{L^1}<1$. Since $\eta(0, 0)=0$ and $\eta(0, \cdot)$ is continuous, there exists a maximal time $T_1\in (0, T]$ such that  $ \eta(0, s)<1$  for  all $s\in [0, T_1)$. Consequently, in \eqref{growth:eta} we have $\varphi(\eta(0, s))=\eta(0, s)[1-\ln(\eta(0, s))]$ provided that $s\in [0, T_1)$. An application of Osgood's lemma \cite[Lemma 3.4]{BCD} yields
\bq\label{Osgood:bound}
\eta(0, r)\le e^{1-\exp(-C|r|\|\om_0^1\|_{L^\infty})}\big(CT\| \om_0^1-\om_0^2\|_{L^1}\big)^{\exp(-C|r|\|\om_0^1\|_{L^\infty})}\quad\forall r\in [0, T_1].
\eq
Using \eqref{Osgood:bound} with $r=T_1$ we find that
if \bq\label{smallcd}
CT\| \om_0^1-\om_0^2\|_{L^1}<e^{1-\exp(CT\|\om_0^1\|_{L^\infty})}
\eq
then $\eta(0, T_1)<1$, and hence $T_1=T$ and \eqref{Osgood:bound} holds for all $r\in [0, T]$. By the same argument, we obtain \eqref{Osgood:bound} for all $r\in [-T, T]$.  Since $\varphi$ is increasing, inserting \eqref{Osgood:bound} into the right-hand side of \eqref{growth:eta}, we deduce 
\bq\label{Osgood:2}
\eta(t, r)\le  \Phi_{T, \|\om_0^1\|_{L^\infty}}(\| \om_0^1-\om_0^2\|_{L^1})+C\|\om_0^1\|_{L^\infty}\left|\int_t^r \varphi(\eta(t, s))ds\right|
\eq
for all $t, r\in [-T, T]$, where
\[
\Phi_{T, \|\om_0^1\|_{L^\infty}}(z)=CTz +CT\|\om_0^1\|_{L^\infty}\varphi\left(e\big(CTz\big)^{\exp(-CT\|\om_0^1\|_{L^\infty})}\right)
\]
with $C$ depending only on $\Omega$.  Clearly $\Phi(z)\to 0$ as $z\to 0$. Similarly to \eqref{Osgood:bound}, we can apply Osgood's lemma to \eqref{Osgood:2} and obtain
\bq\label{est:etatr}
\eta(t, r)\le e^{1-\exp(-C|t-r|\|\om_0^1\|_{L^\infty})}\Phi_{T, \|\om_0^1\|_{L^\infty}}(\| \om_0^1-\om_0^2\|_{L^1})^{\exp(-C|t-r|\|\om_0^1\|_{L^\infty})}
\eq
for all $t, r\in [-T, T]$ provided that 
\bq\label{smallcd:2}
\Phi_{T, \|\om_0^1\|_{L^\infty}}(\| \om_0^1-\om_0^2\|_{L^1})<e^{1-\exp(2CT\|\om_0^1\|_{L^\infty})}.
\eq
 By virtue of  \eqref{est:etatr},  \eqref{est:1} yields
\bq\label{est:2}
\begin{aligned}
&\| \om^1(t)-\om^2(t)\|_{L^p}^p\le 2^{p-1}\| \om^1_0-\om^2_0\|_{L^p}^p+2^{3p-2}\| \om_0^1*\rho_\eps-\om_0^1\|_{L^1}\\
&\quad +C\|\om_0^1*\rho_\eps\|^p_{\dot C^{1/p}}\Phi_{T, \|\om_0^1\|_{L^\infty}}(\| \om_0^1-\om_0^2\|_{L^1})^{\exp(-CT\|\om_0^1\|_{L^\infty})}
\end{aligned}
\eq
for all $t\in [-T, T]$, where $C$ depends only on $\Om$. 

To obtain Theorem \ref{theo}, let $\om_0$, $\om_0^n\in L^\infty(\Om)$ such that $(\om_0^n)_n$  converges to $\om_0$ in $L^p(\Om)$. For $n\ge N$ sufficiently large, the smallness conditions \eqref{smallcd} and \eqref{smallcd:2} hold for $\om_0-\om_0^n$, so that \eqref{est:2} holds for $S_t(\om_0)-S_t(\om_0^n)$. In \eqref{est:2}, taking the supremum over $t\in [0, T]$, then letting $n\to \infty$ followed by $\eps \to 0$ , we conclude that 
\bq\label{est:4}
\lim_{n\to \infty}\sup_{t\in [-T, T]}\| S_t(\om_0^n)-S_t(\om_0)\|_{L^p(\Om)}=0.
\eq
The proof of Theorem \ref{theo} is complete.
\section{Proof of  Theorem \ref{theo:1}}
Assume that $\om_0^n \overset{\ast}{\rightharpoonup}  \om_0$ in $L^\infty(\Om)$. Let $(\om^n, u^n, X^n_t)$ (resp. $(\om, u, X_t)$) be the Yudovich solution of \eqref{Euler} with initial data $\om_0^n$ (resp. $\om_0$). Fix $T>0$ arbitrary. With $M=\sup_{n}\|\om_0^n\|_{L^\infty}<\infty$ we have $\sup_n \| \om^n\|_{L^\infty(\Omega\times (-T, T))}\le M$. Thus there exists a subsequence $\om^{n_k}\wsc \om^\infty$ in $L^\infty(\Omega\times (-T, T))$. Define
\begin{align*}
&u^\infty(t)=K*\om^\infty(t),\\
&\frac{d}{dt}X^\infty_t(x)=u^\infty(X^\infty_t(x), t),\quad X^\infty_0(x)=x\quad\forall x\in \ol{\Om}.
\end{align*}
Note that $u^\infty$ is divergence-free and Log-Lipschitz, whence $X^\infty_t$ is measure-preserving.  We claim that 
\bq\label{claim}
\om^\infty(x)=\om_0(X^\infty_{t, 0}(x)).
\eq
{\bf 1.} To prove  \eqref{claim}, we first use the $L^\infty$ bound \eqref{u:uni} to have
\[
 \sup_n\left|\frac{d}{dt}X^n_t(x)\right|\le CM\quad \forall x\in \Om,\quad C=C(\Om).
\]
 Recall in addition from Theorem \ref{Yu} that each $X^n_t$ is H\"older continuous with exponent  
 \[
 \exp(-C|t|\|\om_0\|_{L^\infty(\Om)})\ge \exp(-CMT),\quad C=C(\Om).
 \] 
  It follows that the sequence $X^n_t(x)$ is uniformly bounded in $C^{\alpha}(\overline{\Om}\times [-T, T])$ for some $\alpha=\alpha(M, T, \Om)$. By the Arzel\`a-Ascoli theorem, the subsequence $X^{n_k}$ has a subsequence  $X^{n_{k_\ell}}\to Y$ in $C(\overline{\Om}\times [-T, T])$. Using this strong convergence, we now prove that 
    \bq\label{prop:0}
\om^{n_k}_0(X^{n_{k_\ell}}_{t, 0}(x))\wsc \om_0(Y_{t, 0}(x))\quad\text{in } L^\infty(\Omega\times (-T, T)),\quad Y_t(x)\equiv Y(x, t).
  \eq
  Indeed, for any $f\in C(\overline{\Om}\times[-T, T])$, we have
  \begin{align*}
\int_{-T}^T\int_\Om \om_0^{n_{k_\ell}}(X^{n_{k_\ell}}_{t, 0}(x))f(x, t)dxdt
&= \int_{-T}^T\int_\Om \om_0^{n_{k_\ell}}(x)f(X^{n_{k_\ell}}_t(x), t)dxdt\\
&=\int_{-T}^T\int_\Om \om_0^{n_{k_\ell}}(x)f(Y_t(x), t)dxdt\\
&\quad+\int_{-T}^T\int_\Om \om_0^{n_{k_\ell}}(x)[f(X^{n_{k_\ell}}_t(x), t)-f(Y_t(x), t)]dxdt\\
&:=I_1+I_2.
  \end{align*}
Since $\om_0^n\wsc \om$ in $L^\infty(\Om)$ and $f(Y_t(\cdot), t)\in L^1(\Om)$, 
\[
\lim_{\ell\to \infty}\int_\Om \om_0^{n_{k_\ell}}(x)f(Y_t(x), t)dx=\int_\Om \om_0(x)f(Y_t(x), t)dx.
\]
In addition, we have
\[
\left|\int_\Om \om_0^{n_{k_\ell}}(x)f(Y_t(x), t)dx\right| \le M\|f(\cdot, t)\|_{L^1}\in L^1((-T, T)),
\]
so that the dominated convergence theorem gives
\[
\lim_{\ell\to \infty}I_1=\int_{-T}^T\int_\Om \om_0(x)f(Y_t(x), t)dxdt.
\]
Since $X^n\to Y$ in $C(\ol{\Om}\times [-T, T])$ and $|\om_0^{n_{k_\ell}}(x)|\le M$, $I_2$ converges to $0$ by uniform convergence on the compact set $\ol{\Om} \times [-T, T]$.  Consequently, 
\bq\label{prop:1}
\begin{aligned}
\lim_{\ell\to \infty}\int_{-T}^T\int_\Om \om_0^{n_{k_\ell}}(X^{n_{k_\ell}}_{t, 0}(x))f(x, t)dxdt&=\int_{-T}^T\int_\Om \om_0(x)f(Y_t(x), t)dx\\
&=\int_{-T}^T\int_\Om \om_0(Y_{t, 0}(x))f(x, t)dx.
\end{aligned}
\eq
Since $C(\overline{\Om}\times[-T, T])$ is dense in $L^1(\Om\times(-T, T))$ and $\om_0^{n}(X^{n}_{t, 0}(x))$ is uniformly bounded in $L^\infty(\Om\times (-T, T))$, \eqref{prop:1} implies \eqref{prop:0}. On the other hand, $ \om^{n_k}_0(X^{n_{k_\ell}}_{t, 0}(x))=\om^{n_{k_\ell}}(x, t) \wsc \om^\infty(x, t)$, so that \eqref{prop:0} implies 
\bq\label{prop:3}
\om^\infty(x, t)=\om_0(Y_{t, 0}(x)).
\eq
Thus \eqref{claim} would follow from \eqref{prop:3} provided that
\bq\label{prop:4}
Y_t(x)=X^\infty_t(x).
\eq
To prove \eqref{prop:4} we start with 
 \bq\label{prop:5}
 \frac{d}{dt}X^n_t(x)=u^n(X^n_t(x), t),\quad u^n(t)=K*\om^n(t).
 \eq
 Using that $\om^{n_k}\wsc \om^\infty$ in $L^\infty(\Omega\times (-T, T))$ and $\int_\Om |K(x, y)|dy\le C(\Om)$, we deduce $u^{n_k}\wsc u^\infty$ in $L^\infty(\Omega\times (-T, T))$. Arguing as in the proof of \eqref{prop:0} we obtain 
 \[
 u^{n_k}(X^{n_k}_t(x), t)\wsc  u^\infty(Y_t(x), t)\quad\text{in } L^\infty(\Omega\times (-T, T)),
 \]
 whence \eqref{prop:5} gives $\frac{d}{dt}Y_t(x)= u^\infty(Y_t(x), t)$. Therefore, $Y_t(x)=X^\infty_t(x)$ by the uniqueness of trajectories generated by Log-Lipschitz velocity fields. This finishes the proof of \eqref{prop:4} and hence of \eqref{claim}.
  
 {\bf 2.} With \eqref{claim} established, the triple $(\om^\infty, u^\infty, X^\infty_t)$ is a Yudovich solution of \eqref{Euler} with initial data $\om^\infty\vert_{t=0}=\om_0$. By the uniqueness part in Theorem \ref{Yu}, $(\om^\infty, u^\infty, X^\infty_t)\equiv (\om, u, X_t)$. In fact, the above argument shows that every subsequence of $\om^n$ has a subsequence converging weakly-$*$ to $\om$ in $L^\infty(\Omega\times (-T, T))$. It follows that the entire sequence $\om^n\wsc \om$  in $L^\infty(\Omega\times (-T, T))$.

 Now for each $t\in \Rr$, $\om^n(t)$ is well-defined in $L^\infty(\Om)$ by virtue of Lemma \ref{lemm}. Moreover, $\| \om^n(t)\|_{L^\infty}\le M$, whence  $\om^{n_k}\wsc  h(t)$ in $L^\infty(\Om)$ for some subsequence $n_k$ which a priori depends on $t$. For any $f\in C(\ol{\Om})$ we have
\begin{align*}
\int_\Om \om^n(x, t)f(x)dx&= \int_\Om \om^n_0(x)f(X^n_t(x))dx\to  \int_\Om \om_0(x)f(X_t(x))dx
\end{align*}
in views of the facts that $\om_0^n\wsc \om_0$ in $L^\infty(\Om)$ and $X^n_t\to X^\infty_t\equiv X_t$ in $C(\ol{\Om})$. It follows that 
\[
\int_\Om h(x, t)f(x)dx= \int_\Om \om_0(x)f(X_t(x))dx=\int_\Om \om_0(X_{t, 0}(x))f(x)dx
\]
 and thus $h(x, t)= \om_0(X_{t, 0}(x))=\om(x, t)$ a.e. $x\in \Om$. In fact, we have proved that every subsequence of $\om^n(\cdot, t)$ has a subsequence converging weakly-$*$ to $\om(\cdot, t)$ in $L^\infty(\Om)$. Therefore, the entire sequence $\om^n(\cdot, t)\wsc \om(\cdot, t)$  in $L^\infty(\Om)$. This concludes the proof of Theorem \ref{theo:1}.
\section{Proof of Theorem \ref{theo:2}}
We follow closely the proof of Theorem \ref{theo:1} and use the same notation whenever possible. We note that $\om^n(t)$ has compact support for all $t\in \Rr$ but $X^n_t$ does not in general. Since $X^n$ is uniformly  H\"older continuous on any compact set of $\Rr^2_x\times \Rr_t$, any subsequence $n_k$ has a  subsequence $n_{k_\ell}$ such that 
\bq\label{conv:BR}
\forall R>0,\quad X^{n_{k_\ell}}\to Y\quad\text{in } C(\ol{B_R}\times [-T, T])
\eq
  by the  Arzel\`a-Ascoli theorem and a diagonal procedure. Here $B_R$ denotes the ball of radius $R$ centered at $0\in \Rr^2$.  To prove \eqref{prop:0} we take $f\in C_c(\Rr^2\times (-T, T))$, a dense subspace of $L^1(\Om\times (-T, T))$, and only consider $I_2$ since the above argument for $I_1$ does not make use of the boundedness of $\Om$. If $\supp f\subset B_R\times (-T, T)$ then
  \[
|I_2|\le M\int_{-T}^T\int_{X^{n_{k_\ell}}_{t, 0}(B_R)\cup Y_{t, 0}(B_R)} |f(X^{n_{k_\ell}}_t(x), t)-f(Y_t(x), t)|dxdt,
\]
where $M=\sup_n\| \om_0^n\|_{L^\infty}$. We have
\[
\| u^n\|_{L^\infty}\le C\| \om^n\|_{L^1\cap L^\infty}= C\|\om_0^n\|_{L^1\cap L^\infty}\le CN,
\]
where $N:=\sup_n\|\om_0^n\|_{L^1\cap L^\infty}+\|\om_0\|_{L^1\cap L^\infty}$. This implies
\[
X^{n_{k_\ell}}_{t, 0}(B_R)\cup Y_{t, 0}(B_R) \subset B_{R+TCN}\quad\forall |t|\le T,
\]
whence
\[
|I_2|\le M\int_{-T}^T\int_{B_{R+TCN}} |f(X^{n_{k_\ell}}_t(x), t)-f(Y_t(x), t)|dxdt.
\]
Therefore, $\lim_{\ell\to\infty} I_2=0$ by the uniform convergence \eqref{conv:BR} on the compact set $\ol{B_{R+TCN}}\times [-T, T]$. This yields \eqref{prop:0}. 

Regarding \eqref{prop:5}, we prove that $\om^{n_k}\wsc \om^\infty$ in $L^\infty(\Rr^2\times (-T, T))$ implies $u^{n_k}\wsc u^\infty:=K*\om^\infty$ in $L^\infty(\Rr^2\times (-T, T))$. Note that the Biot-Savart kernel for $\Rr^2$ is 
\[
K(x, y)\equiv K(x-y),\quad K(x)=\frac{(-x_2, x_1)}{2\pi |x|^2}
\]
and $K$ does not belong to $L^1(\Rr^2)$. For any $f\in C_c(\Rr^2\times (-T, T))$ we have 
\bq\label{om:to:u}
\begin{aligned}
\int_{-T}^T\int_{\Rr^2}u^{n_k}(x, t)f(x, t)dxdt&=\int_{-T}^T \int_{\Rr^2}\om^{n_k}(y, t)\int_{\Rr^2}K(x-y)f(x, t)dxdydt\\
&:=\int_{-T}^T \om^{n_k}(y, t)g(y, t)dydt.
\end{aligned}
\eq
Since $K\notin L^1(\Rr^2)$, we do not have $g(t)\in L^1(\Rr^2)$ to use the weak-$*$ convergence of $\om^{n_k}$. On the other hand, upon splitting the $x$-integration in $g$ into $|x-y|\le 1$ and $|x-y|>1$ and applying suitable Young inequalities, we deduce $\| g(t)\|_{L^3}\le C\| f(t)\|_{L^1\cap L^3}$ and thus $g\in L^\infty((-T, T); L^3)$. Now, $\om^n$ is uniformly bounded in $L^\infty(\Rr; L^\tdm)$ by interpolation, whence $\om^{n_k}\wc \om^\infty$ in $L^\tdm(\Rr^2\times (-T, T))$. It then follows from \eqref{om:to:u} that
\begin{align*}
\lim_{k\to \infty}\int_{-T}^T\int_{\Rr^2}u^{n_k}(x, t)f(x, t)dxdt&=\int_{-T}^T\int_{\Rr^2} \om^\infty(y, t)g(y, t)dydt\\
&=\int_{-T}^T\int_{\Rr^2}u^\infty (x, t)f(x, t)dxdt
\end{align*}
for all $f\in C_c(\Rr^2\times (-T, T))$. Using this, the fact that $u_n$ is uniformly bounded in $L^\infty(\Rr^2)$ and a density argument, we conclude $u^{n_k}\wsc u^\infty$ in $L^\infty(\Rr^2\times (-T, T))$. The remainder of the proof follows along the same lines of the proof of Theorem \ref{theo:1}.

\vspace{.1in}
\noindent{\bf{Acknowledgment.}} 
The work of HQN was partially supported by NSF grant DMS-1907776. I would like to thank Theodore D.  Drivas for pointing out the reference \cite{Sverak} and for interesting discussions on Yudovich theory.


\vspace{1mm}

\begin{thebibliography}{10}
\small

\bibitem{BCD}H. Bahouri, J-Y Chemin, and R. Danchin.
\newblock {\em Fourier analysis and nonlinear partial differential equations},
  volume 343 of {\em Grundlehren der Mathematischen Wissenschaften [Fundamental
  Principles of Mathematical Sciences]}.
\newblock Springer, Heidelberg, 2011.


\bibitem{CDE} P. Constantin,  T. D. Drivas, and  T. M. Elgindi. Inviscid Limit of Vorticity Distributions in the Yudovich Class. {\it Commun. Pure Appl. Math.}, to appear.

\bibitem{MB}  A. J. Majda, and A. L. Bertozzi. {\it Vorticity and incompressible flow}. Cambridge Texts in Applied Mathematics, 27. Cambridge University Press, Cambridge, 2002. 

\bibitem{MP} C. Marchioro, and M. Pulvirenti. {\it Mathematical Theory of Incompressible Nonviscous Fluids}, Applied Mathematical Sciences, vol. 96, Springer-Verlag, New York, 1994.

\bibitem{Sverak} V. Sverak. Course notes on ``Selected Topics in Fluid Mechanics". http://www-users.math.umn.edu/~sverak/

\bibitem{Yu} V. I. Yudovich. Non-stationary flows of an ideal incompressible fluid. (Russian) Zh. Vychisl. Mat. Mat. Fiz. 3 (1963), 1032--1066.
 
\end{thebibliography}
\end{document}